\documentclass[a4paper, 10pt]{amsart}
\usepackage{amsmath}
\usepackage{amssymb, color, hyperref}

\theoremstyle{plain}
\newtheorem{theorem}{\bf Theorem}[section]
\newtheorem{proposition}[theorem]{\bf Proposition}
\newtheorem{lemma}[theorem]{\bf Lemma}
\newtheorem{corollary}[theorem]{\bf Corollary}

\theoremstyle{definition}

\newtheorem{remark}[theorem]{\bf Remark}

\newcommand{\N}{\mathbb N}
\newcommand{\Z}{\mathbb Z}
\newcommand{\R}{\mathbb R}

\newcommand{\BF}{\text{\rm BF}}
\newcommand{\ACC}{\text{\rm ACC  }}

\newcommand{\red}{{\text{\rm red}}}

\numberwithin{equation}{section}

\begin{document}

\title[Sets of lengths]{Sets of lengths  in \\ atomic unit-cancellative finitely presented monoids}

\address{Institute for Mathematics and Scientific Computing\\ University of Graz, NAWI Graz\\ Heinrichstra{\ss}e 36\\ 8010 Graz, Austria }
\email{alfred.geroldinger@uni-graz.at}
\urladdr{http://imsc.uni-graz.at/geroldinger}

\address{Department of Mathematical Sciences \\ University of Texas at El Paso \\ 500 W. University Ave \\ El Paso, Texas 79968-0514, USA }
\email{eschwab@utep.edu}
\urladdr{http://www.math.utep.edu/Faculty/schwab/}

\author{Alfred Geroldinger and Emil Daniel Schwab}

\thanks{This work was supported by the Austrian Science Fund FWF, Project Number P28864-N35}



\begin{abstract}
For an element $a$ of a monoid $H$, its set of lengths $\mathsf L (a) \subset \N$ is the set of all positive  integers $k$ for which there is a factorization $a=u_1 \cdot \ldots \cdot u_k$ into $k$ atoms. We study the system $\mathcal L (H) = \{\mathsf L (a) \mid a \in H \}$ with a focus on the unions $\mathcal U_k (H) \subset \N$ which are the unions of all sets of lengths containing a given $k \in \N$. The Structure Theorem for Unions -- stating that for all sufficiently large $k$, the sets $\mathcal U_k (H)$ are almost arithmetical progressions with the same difference and global bound -- has found much attention for commutative monoids and domains. We show that it holds true for the not necessarily commutative monoids in the title satisfying suitable algebraic finiteness conditions. Furthermore, we give an explicit description of the system of sets of lengths of monoids $B_{n} = \langle a,b \mid  ba=b^{n} \rangle$  for  $n \in \N_{\ge 2}$. Based on this description, we  show that the monoids $B_n$ are not transfer Krull, which implies that their systems $\mathcal L (B_n)$ are distinct from systems of sets of lengths of commutative Krull monoids and others.
\end{abstract}

\maketitle

\bigskip
\section{Introduction} \label{1}
\bigskip

By an atomic unit-cancellative monoid, we mean an associative semigroup with unit element where every non-unit can be written as a finite product of atoms (irreducible elements) and where  equations of the form $au=a$ or $ua=a$ imply that $u$ is a unit. Let $H$ be an atomic unit-cancellative monoid. If $a = u_1 \cdot \ldots \cdot u_k$, with $a \in H$ and atoms $u_1, \ldots, u_k \in H$, then $k$ is called a factorization length of $a$ and the set $\mathsf L (a) \subset \N$ of all possible factorization lengths is called the set of lengths of $a$. The system $\mathcal L (H) = \{\mathsf L (a) \mid a \in H \}$ of all sets of lengths (for convenience, one defines $\mathsf L (a) = \{0\}$ for units $a \in H$) is a well-studied means of describing the non-uniqueness of factorizations of $H$.  If there is an $a \in H$ with $|\mathsf L (a)|>1$, say $k, \ell \in \mathsf L (a)$, then for all $N \in \N$, $\mathsf L (a^N) \supset \{(N-i)k + i \ell \mid i \in [0,N]\}$ and hence $\mathsf |\mathsf L (a^N)|>N$. Thus $\mathcal L (H)$   either consists of singletons only or it contains arbitrarily large sets. For every $k \in \N$, let $\mathcal U_k (H)$ denote the set of all $\ell \in \N$ for which there is an equation of the form  $u_1 \cdot  \ldots \cdot  u_k = v_1 \cdot \ldots \cdot v_{\ell}$ with atoms $u_1, \ldots, u_k,v_1, \ldots , v_{\ell} \in H$. Thus $\mathcal U_k (H)$ is the union of all sets of lengths containing $k$.

Systems of sets of lengths, together with all invariants derived from them, such as unions, sets of distances,  and more, are well-studied invariants in factorization theory. For a long time the focus was on commutative and cancellative monoids which mainly stem from ring theory, such as monoids of non-zero elements of integral domains or monoids of invertible ideals (see \cite{An97, Ge-HK06a, Fo-Ho-Lu13a, C-F-G-O16}).
Recently, first steps were made to study the arithmetic of commutative but not necessarily cancellative monoids (see \cite{Ch-An13a, Ch-Sm13a, F-G-K-T17}).

Although various concepts of unique factorization in non-commutative rings have been studied for decades (see \cite{Sm16a} for a  survey), a systematic investigation  of arithmetic phenomena in non-commutative rings has been started only in recent years (\cite{Sm13a, Ba-Ba-Go14, Ba-Sm15, Ba-Je16a,  Be-He-Le17, Ba-Sm18}). In many of these papers the authors construct so-called (weak) transfer homomorphisms from the non-commutative ring $R$ under consideration to a commutative monoid $H$  such that many arithmetical phenomena of $R$ and $H$ coincide and, in particular, $\mathcal L (R)= \mathcal L (H)$ holds. To mention one of these deep results explicitly, consider a bounded hereditary noetherian prime ring $R$: if every stably free left $R$-ideal is free, then there are a commutative Krull monoid $H$ and a weak transfer homomorphism $\theta \colon R \to H$ implying  that $\mathcal L (R) = \mathcal L (H)$ (\cite[Theorem 4.4]{Sm17a}).

In the present note we study atomic unit-cancellative monoids which are finitely presented. Under mild algebraic finiteness conditions, we show that all unions $\mathcal U_k (H)$ are finite and that, apart from  globally bounded initial and end parts, they are arithmetical progressions. Thus $\mathcal L (H)$ satisfies the Structure Theorem for Unions (see Lemma \ref{3.1} and Theorem \ref{3.6}).

Section \ref{4} is devoted to the monoids $B_{n} = \langle a,b \mid  ba=b^{n} \rangle$  for  $n \in \N_{\ge 2}$.  They are M\"obius monoids anti-isomorphic to the semigroups $S_{2,n}$  from \cite{Sh89a} with a unit element adjoined, playing an important role in  M\"obius inversion and in the study of monoids with one defining relation generating varieties of finite axiomatic rank (\cite{Sh89a, Sc-Ha08a, Schw15a, Schw16a}). It is easy to see that the monoids $B_n$ are atomic, unit-cancellative, and right-cancellative, but not left-cancellative. We provide an explicit description of the system $\mathcal L (B_n)$  and of the unions $\mathcal U_k (B_n)$ for all $k \ge 2$  (Theorem \ref{4.2} and Corollary \ref{4.3}). Such explicit descriptions are very rare in the literature (cf. \cite[Corollary 16]{Ch-Sm13a}, \cite{Ge-Sc-Zh17b}). They allow us to show that there is no weak transfer homomorphism $\theta \colon B_n \to H$ where $H$ is any commutative Krull monoid (Corollary \ref{4.4}).  Furthermore, we prove that the system $\mathcal L (B_n)$ is closed under set addition and hence $\mathcal L (B_n)$ is a reduced atomic unit-cancellative monoid with set addition as operation (Theorem \ref{4.5}).

\bigskip
\section{Preliminaries} \label{2}
\bigskip

We denote by $\N$ the set of positive integers. For a real number $x \in \R_{\ge 0}$, we denote by $\lfloor x \rfloor \in \N_0$ the largest integer which is smaller than or equal to $x$. For  $a, b \in \Z$, $[a,b] = \{x \in \Z \mid a \le x \le b\}$ denotes the discrete interval. Let $L, L'  \subset \Z$ be subsets of the integers. Then $L+L' = \{\ell + \ell' \mid \ell \in L, \ell' \in L'\}$ is their sumset with $m+L = \{m\}+L$ for every $m \in \Z$. For $d \in \N$, $d \cdot L = \{ da \mid a \in L \}$ denotes the {\it dilation} of $L$ by $d$. Thus, for $d \in \N, q \in \N_0$, and $m \in \Z$, $m+ d \cdot [0, q] = \{m, m+d, \ldots, m+qd \}$ is an arithmetical progression with difference $d$.
A positive integer $d \in \N$ is said to be a distance of $L$ if there is an $a \in L$ such that $[a, a+d] \cap L = \{a, a+d\}$ and $\Delta (L) \subset \N$ is the {\it set of distances} of $L$. If $L \subset \N_0$, then the {\it elasticity} $\rho (L)$ is defined as $\rho (L) = \sup (L \cap \N)/\min (L  \cap \N) \in \mathbb Q_{\ge 1} \cup \{\infty\}$ if $L \cap \N \ne \emptyset$ and $\rho (L)=1$ if  $L \cap \N = \emptyset$.

Let $\mathcal L$ be a family of subsets of $\N_0$. For each $k \in \N$, we set
\[
\mathcal U_k ( \mathcal L) = \bigcup_{k \in L, L \in \mathcal L } \ L \quad \subset \N  \quad \text{and} \quad \rho_k ( \mathcal L ) = \sup \, \mathcal U_k (\mathcal L) \,.
\]
Furthermore, we call
\begin{itemize}
\item $\Delta ( \mathcal L ) = \bigcup_{L \in \mathcal L} \Delta (L) \subset \N$ \ the {\it set of distances} of $\mathcal L$,

\smallskip
\item $\rho ( \mathcal L) = \sup \{\rho (L) \mid L \in \mathcal L \} \in \R_{\ge 1} \cup \{\infty\}$ \ the {\it elasticity} of $\mathcal L $, and
\end{itemize}
we say that $\mathcal L$ has {\it accepted elasticity} if there is some $L \in \mathcal L$ such that $\rho (L) = \rho (\mathcal L) < \infty$.

\smallskip
By a {\it monoid} we mean an associative semigroup with a unit element. If not stated otherwise, we use multiplicative notation. Let $H$ be a monoid with unit element $1_H=1$. We denote by $H^{\times}$ the group of invertible elements of $H$ and say that $H$ is reduced if $H^{\times} = \{1\}$. We say that $H$ is (left and right) unit-cancellative if the following two properties are satisfied:
\begin{itemize}
\item  If $a, u \in H$ and $a = au$, then $u \in H^{\times}$,

\item  If $a, u \in H$ and $a = ua $, then $u \in H^{\times}$.
\end{itemize}
Clearly, every cancellative monoid is unit-cancellative.
Unit-cancellativity is a frequently studied property, by many authors and under  many different names (the corresponding concept in ring theory is called pr\'esimplifiable; it was introduced by Bouvier and further studied by D.D. Anderson et al. \cite{An-VL96a, An-Al17a}).
An element $u \in H$ is said to be {\it irreducible} (or an {\it atom}) if $u \notin H^{\times}$ and an equation $u = ab $ with $a, b \in H$ implies that $a \in H^{\times}$ or $b \in H^{\times}$. We denote by $\mathcal A (H)$ the set of atoms of $H$ and we say that $H$ is atomic if every non-unit can be written as a finite product of atoms of $H$.

For a set $P$, let $\mathcal F^* (P)$ be the free monoid with basis $P$ and let $\mathcal F (P)$ be the free abelian monoid with basis $P$. We denote by $|\cdot | \colon \mathcal F^* (P) \to \N_0$ the homomorphism mapping each word onto its length. Similarly, if
\[
a = \prod_{p \in P}p^{\mathsf v_p (a)} \in \mathcal F (P), \quad \text{where} \quad \mathsf v_p \colon \mathcal F (P) \to \N_0 \quad \text{is the $p$-adic exponent} \,,
\]
then $|a|=\sum_{p \in P} \mathsf v_p (a) \in \N_0$ is the length of $a$.
Let $D$ be a monoid. A submonoid $H \subset D$ is said to be {\it saturated} if $a \in D$, $b \in H$, and ($ab \in H$ or $ba \in H$) imply that $a \in H$ (instead of saturated also the terms {\it full} or {\it grouplike} are used, see \cite[page 64]{Je-Ok07a}). A commutative monoid $H$ is Krull if its associated reduced monoid is a saturated submonoid of a free abelian monoid (\cite[Theorem 2.4.8]{Ge-HK06a}).

Let $H$ be an atomic unit-cancellative monoid. If $a = u_1 \cdot \ldots \cdot u_k \in H$, where $k \in \N$ and $u_1, \ldots, u_k \in \mathcal A (H)$, then $k$ is called a factorization length and
\[
\mathsf L_H (a) = \mathsf L (a) = \{ k \in \N \mid a \ \text{has a factorization of length} \ k \} \subset \N
\]
is called the {\it set of lengths} of $a$. For a unit $\epsilon \in H^{\times}$ we set $\mathsf L (\epsilon) = \{0\}$. Then
\[
\mathcal L (H) = \{ \mathsf L (a) \mid a \in H \}
\]
denotes the {\it system of sets of lengths} of $H$. We say that $H$ is
\begin{itemize}
\item a {\it \BF-monoid} (a bounded factorization monoid) if $\mathsf L (a)$ is finite and nonempty for all $a \in H$,

\item {\it half-factorial} if $|\mathsf L (a)|=1$ for all $a \in H$.
\end{itemize}
If $H$ is not half-factorial, then there is an $a \in H$ with $|\mathsf L (a)|>1$ whence $\mathsf L (a^n) \supset \mathsf L (a) + \ldots + \mathsf L (a)$ has more than $n$ elements for every $n \in \N$. We denote by
\begin{itemize}
\item $\mathcal U_k (H) = \mathcal U_k \big( \mathcal L (H) \big)$ for every $k \in \N$  the {\it union of sets of lengths} (containing $k$), by

\smallskip
\item $\Delta (H) = \Delta \big( \mathcal L (H) \big)$  {\it the set of distances} of $H$, and by

\smallskip
\item $\rho (H) = \rho \big( \mathcal L (H) \big)$ is the {\it elasticity} of $H$.
\end{itemize}

\bigskip
\section{Reduced atomic  unit-cancellative monoids} \label{3}
\bigskip

In this section we study sets of lengths of reduced atomic unit-cancellative  monoids.  Under suitable additional algebraic finiteness conditions we show that their system of sets of lengths satisfies the Structure Theorem for Unions (Theorem \ref{3.6}).
To start with, we recall some concepts needed to formulate the Structure Theorem for Unions.

A non-empty subset $L \subset \N_0$ is called an {\it almost arithmetical progression} (AAP for short) if
\[
L \subset \min L + d \Z \quad \text{and} \quad L \cap [M+ \min L, -M + \sup L] \ \text{is an arithmetical progression with difference} \ d \,,
\]
where $d \in \N$,  $M \in \N_0$, and with the conventions that arithmetical progressions are non-empty and that $[M+ \min L, -M + \sup L] = \N_{\ge M+ \min L}$ if $L$ is infinite.
A family $\mathcal L$ of subsets of $\N_0$ is said to
\begin{itemize}
\item be {\it directed} if $1 \in L$ for some $L \in \mathcal L$ and, for all $L_1, L_2 \in \mathcal L$, there is $L^\prime \in \mathcal L$ with $L_1 + L_2 \subset L^\prime$;

\item satisfy the {\it Structure Theorem for Unions} if there are $d \in \N$ and $M \in \N_0$ such that $\mathcal U_k (\mathcal L)$ is an AAP with difference $d$ and bound $M$ for all sufficiently large $k \in \N$.
\end{itemize}

\smallskip
The next result provides a characterization of when the Structure Theorem for Unions holds in the setting of directed families of subsets of the non-negative integers.

\smallskip
\begin{lemma} \label{3.1}
Let $\mathcal L$ be a  directed family of subsets of $\N_0$ such that $\Delta(\mathcal L)$ is finite nonempty and set $d = \min \Delta(\mathcal L)$.
Let $\ell \in \N$ such that $\{\ell, \ell + d\} \subset L$ for some $L \in \mathcal L$. Then $q = \frac{1}{d} \max \Delta(\mathcal L)$  is a non-negative integer and the following statements are equivalent{\rm \,:}
\begin{enumerate}
\item[(a)] $\mathcal L$ satisfies the  Structure Theorem for Unions.
\item[(b)] There exists $M \in \N$ such that \ $\mathcal U_k (\mathcal L) \cap [ \rho_{k - \ell q}(\mathcal L) + \ell q, \rho_k (\mathcal L) - M ]$ is either empty or an arithmetical progression with difference $d$ for all sufficiently large $k$.
\end{enumerate}
\end{lemma}

A proof of Lemma \ref{3.1} together with a variety of properties of directed families and of consequences of the Structure Theorem for Unions can be found in \cite[Section 2]{F-G-K-T17}. In particular, we recall that if the Structure Theorem holds and $\Delta (\mathcal L)$ is nonempty, then the unions are AAPs with difference $\min \Delta (\mathcal L)$. We now analyze Condition (b) in Lemma \ref{3.1}. If there is some $\ell \in \N$ such that $\rho_{\ell} (\mathcal L)=\infty$, then $\rho_k (\mathcal L)=\infty$ for all $k \ge \ell$ whence the  intersection given in Condition (b) is empty. Suppose that $\rho_k(\mathcal L) < \infty$ for all $k \in \N$. If there is an $M \in \N$ such that $\rho_{k}(\mathcal L) - \rho_{k-1}(\mathcal L) \le M$ for all $k \in \N_{\ge 2}$, then Condition (b) holds (for details see \cite[Section 2]{F-G-K-T17}).

\smallskip
Let $H$ be an atomic unit-cancellative monoid. If $u \in \mathcal A (H)$, then $\mathsf L (u) = \{1\}$. If $a, b \in H$, then $\mathsf L (a) + \mathsf L (b) \subset \mathsf L (ab)$. Thus, if $H \ne H^{\times}$, then $\mathcal L (H)$ is a directed family and  the characterization given in Lemma \ref{3.1} applies to the system of sets of lengths $\mathcal L (H)$.

Two elements $a, b \in H$ are said to be associated (we write $a \simeq b$) if $a \in H^{\times} b H^{\times}$. Suppose that $aH^{\times} = H^{\times} a$ for all $a \in H$. Then being associated is a congruence relation on $H$ and $H_{\red} = H/\simeq$ is a reduced atomic monoid. For every $a \in H$, we have $[a]_{\simeq} = aH^{\times}$ and $\mathsf L_H (a) = \mathsf L_{H_{\red}} (aH^{\times})$ whence $\mathcal L (H) = \mathcal L (H_{\red})$. We will formulate all our results for reduced atomic monoids but they also apply to non-reduced monoids as outlined in Remark \ref{3.7}.2.

\smallskip
We now suppose that $H$ is  reduced  atomic and unit-cancellative. We define its factorization monoid as the monoid of formal products of atoms and distinguish between the commutative and the non-commutative case. Thus we call
\[
\mathsf Z (H) = \begin{cases}
                \mathcal F^* ( \mathcal A (H)),  & \  \text{if $H$ is non-commutative} \\
                \mathcal F ( \mathcal A (H)), &  \ \text{if $H$ is commutative} \\
                \end{cases}
\]
the {\it factorization monoid}   of $H$. Then $\pi \colon \mathsf Z (H) \to H$ denotes the canonical epimorphism. If $a \in H$, then $\mathsf Z (a) = \pi^{-1} (a) \subset \mathsf Z (H)$ is the {\it set of factorizations} of $a$ and $\mathsf L (a) = \{ |z| \mid z \in \mathsf Z (a) \} \subset \N_0$ is the set of lengths of $a$, as introduced above. We say that $H$ is an FF-{\it monoid} (a finite factorization monoid) if $\mathsf Z (a)$ is finite  and non-empty for every $a \in H$. Note that every reduced atomic unit-cancellative FF-monoid is a M\"obius monoid (see \cite[Section 2.1]{Schw15a}).
The {\it monoid of relations} is defined as
\[
\sim_H \ = \{ (x,y) \in \mathsf Z (H) \times \mathsf Z (H) \mid \pi (x) = \pi (y) \} \,,
\]
and a {\it distance} on $H$ is a map $\mathsf d \colon \sim_H \to \N_0$ satisfying the following properties for all $z,z', z'' \in \ \sim_H$:
\begin{itemize}
\item[{\bf (D1)}] $\mathsf d (z,z)=0$,
\item[{\bf (D2)}] $\mathsf d (z,z')=\mathsf d (z',z)$,
\item[{\bf (D3)}] $\mathsf d (z,z') \le \mathsf d (z,z'')+\mathsf d (z'',z')$,
\item[{\bf (D4)}] $\mathsf d (xz, xz')=\mathsf d (zy, z'y)= \mathsf d (z,z')$ for all $x,y$, and
\item[{\bf (D5)}] $\big| |z|-|z'| \big| \le \mathsf d (z,z') \le \max \{|z|,|z'|,1\}$.
\end{itemize}
Having distance functions at our disposal, we can introduce the concept of  catenary degrees. For an element $a \in H$, the {\it catenary degree} $\mathsf c_{\mathsf d} (a) \in \N_0 \cup \{\infty\}$ (of $a$ with respect to the distance function $\mathsf d$) is the minimal $N \in \N_0 \cup \{\infty\}$ such that for any two factorizations $z, z'$ of $a$ there are factorizations $z=z_0, z_1, \ldots, z_n=z'$ of $a$ such that $\mathsf d (z_{i-1},z_i) \le N$ for all $i \in [1,n]$. The {\it catenary degree} (in distance $\mathsf d$) of $H$ is
\[
\mathsf c_{\mathsf d} (H) = \sup \{ \mathsf c_{\mathsf d} (a) \mid a \in H \} \in \N_0 \cup \{\infty\} \,.
\]
Property {\bf (D5)} easily implies that $\sup \Delta \big( \mathsf L (a) \big) \le \mathsf c_{\mathsf d} (a)$ for every $a \in H$ and hence (see \cite[Lemma 4.2]{Ba-Sm15})
\begin{equation} \label{basic-inequality}
\sup \Delta (H) \le \mathsf c_{\mathsf d} (H) \,.
\end{equation}

First, suppose that $H$ is commutative (and still reduced atomic unit-cancellative). To recall the usual distance function, consider two factorizations $z,z' \in \mathsf Z (H)$. Then there exist $\ell, m,n \in \N_0$ and $u_1, \ldots, u_{\ell}, v_1, \ldots, v_m, w_1, \ldots, w_n \in \mathcal A (H)$ with $\{v_1, \ldots, v_m\} \cap \{w_1, \ldots, w_n\} = \emptyset$ such that
\[
z= u_1 \cdot \ldots \cdot u_{\ell}v_1 \cdot \ldots \cdot v_m \quad \text{and} \quad z' = u_1 \cdot \ldots \cdot u_{\ell}w_1 \cdot \ldots \cdot w_n \,.
\]
Then the map $\mathsf d \colon \sim_H \to \N_0$, defined by $\mathsf d (z,z') = \max \{m,n\}$,  is a distance function satisfying {\bf (D1)} - ({\bf D5}). If $H$ is commutative and cancellative but not half-factorial, then the stronger inequality  $2 + \sup \Delta (H) \le \mathsf c_{\mathsf d} (H)$ holds (\cite[Chapter 1]{Ge-HK06a}).

Now suppose that $H$ is not commutative. Then the above definition of factorizations and of the associated  factorization monoid coincides with the concept of rigid factorizations studied by Baeth and Smertnig  who also introduced the concept of abstract distance functions as given above (\cite[Remark 3.3]{Sm13a}, \cite{Ba-Sm15}, \cite[Chapter 5]{Sm16a}). A well-studied distance is the  Levenshtein distance, which defines the distance between two elements $z,z' \in \mathsf Z (H)$ as  the minimum number of operations needed to transform $z$ into $z'$, where an operation is a substitution, deletion, or insertion of an atom $a \in \mathcal A (H)$ (see \cite[Section 3]{Ba-Sm15} for this and other distance functions).

 Recall that even monoids presented by a single defining relation need not be atomic (e.g, $H = \langle a, b \mid a = bab \rangle$) and that  there are also finitely generated commutative monoids which are not atomic (e.g., \cite[Theorem 3 and Example 6]{Ro-GS-GG04b}). However, every unit-cancellative monoid satisfying the ACC on principal right ideals and on principal left ideals is atomic (see \cite[Proposition 3.1]{Sm13a} for the cancellative case, \cite[Lemma 3.1]{F-G-K-T17} for the commutative unit-cancellative case, and \cite[Theorem 2.6]{Fa-Tr18a} for the general case). Clearly, the ACC on all right ideals is a much stronger condition, as the next lemma shows.

\medskip
\begin{lemma} \label{3.2}
Let $H$ be a reduced submonoid of a group satisfying the \ACC on right ideals.
\begin{enumerate}
\item  Then $H$ satisfies the \ACC on left ideals and $H$ is atomic with finite set of atoms.

\smallskip
\item If $S \subset H$ is a saturated submonoid, then $S$ satisfies the \ACC on right ideals.
\end{enumerate}
\end{lemma}

\begin{proof}
See \cite[Lemmas 4.1.1 and 4.2.5]{Je-Ok07a}.
\end{proof}

\medskip
\begin{lemma} \label{3.3}
Let $H$ be a reduced atomic cancellative monoid.
\begin{enumerate}
\item The monoid of relations $\ \sim_H \ \subset \mathsf Z (H) \times \mathsf Z (H)$ is a saturated submonoid and a reduced cancellative \BF-monoid.

\smallskip
\item If the  monoid of relations $\ \sim_H \ $ satisfies the \ACC on right ideals, then $\ \sim_H \ $ is finitely generated.

\smallskip
\item If $H$ is commutative and finitely generated, then the monoid of relations $\ \sim_H \ $ satisfies the \ACC on  ideals and $\ \sim_H \ $ is finitely generated.

\end{enumerate}
\end{lemma}

\begin{proof}
1. To show that $\ \sim_H \ $ is a saturated submonoid, let $(x_1, y_1), (x_2, y_2) \in \mathsf Z (H) \times \mathsf Z (H)$ be given. Suppose that $(x_1, y_1) \in \ \sim_H$ and that $(x_1, y_1)(x_2,y_2) \in \ \sim_H$ or $(x_2, y_2)(x_1,y_1) \in \ \sim_H$, say $(x_1, y_1)(x_2,y_2) \in \ \sim_H$. Then  $\pi (x_1)=\pi (y_1) \in H$ and $\pi (x_1) \pi (x_2)= \pi (y_1)\pi (y_2) \in H$. Thus the cancellativity of $H$ implies that $\pi (x_2)=\pi (y_2) \in H$ whence $(x_2, y_2) \in \ \sim_H$.
Since $\mathsf Z (H) \times \mathsf Z (H)$ is a reduced cancellative monoid, $\ \sim_H \ $ is a reduced cancellative  monoid and it is a BF-monoid by \cite[Lemma 2]{Ge16c}.

2. By Lemma \ref{3.2}.1,  the set of atoms $\mathcal A (\sim_H)$ is finite whence \ $\sim_H$ \ is finitely generated.

3. Suppose that $H$ is commutative and finitely generated. Then $\mathsf Z (H)$ is finitely generated by definition and it satisfies the ACC on ideals by \cite[Theorem 3.6]{HK98}. Thus 1. and Lemma \ref{3.2}.2 imply that $\ \sim_H \ $ satisfies the ACC on  ideals, and hence $\ \sim_H \ $ is finitely generated by 2.
\end{proof}

\smallskip
A reduced atomic monoid is said to be {\it finitely presented} if $A = \mathcal A (H)$ is finite and there is a finite set of relations
\[
R \subset \ \sim_H
\]
which generates $\sim_H$ as a congruence relation (this means that one first defines a reflexive, symmetric relation $E \subset \mathsf Z (H) \times \mathsf Z (H)$ by defining $x \sim_E y$ if and only if $x = sut$ and $y = svt$ with $s, t \in \mathsf Z (H)$ and $(u,v) \in R \cup R^{-1} \cup \{(x,x)\mid x \in \mathsf Z (H)\}$, and then one takes the  transitive closure of $E$; for details see \cite[Chapter 1.5]{Cl-Pr64}). Thus, if $\ \sim_H$ is finitely generated as a monoid, then $H$ is finitely presented. As usual we write $H = \langle A \mid R \rangle = \langle A \mid x_1=y_1, \ldots, x_m=y_m \rangle$ if $R = \{(x_1, y_1), \ldots, (x_m,y_m)\} \subset \ \sim_H$.

If an atomic monoid $H$ is presented by homogenous relations (that is, for all $(x,y) \in R$, we have $|x|=|y|$), then $H$ is half-factorial (that is $\Delta (H)= \emptyset$), and conversely (monoids presented by homogenous relations have found wide interest in the study of finitely presented algebras, e.g., \cite{Ce-Je-Ok10, Ok16a}).  Our first result outlines that finitely presented monoids have finite sets of distances.

\medskip
\begin{proposition} \label{3.4}
Let $H$ be a reduced atomic  unit-cancellative and finitely presented monoid with $H \ne \{1\}$, say $H = \langle \mathcal A (H) \mid x_1=y_1, \ldots, x_m=y_m \rangle$, and let $\mathsf d \colon \sim_H \ \to \N_0$ be a distance on $H$.  Then we have
\[
\sup \Delta (H) \le \mathsf c_{\mathsf d} (H) \le \max \{ \mathsf d (x_1 , y_1), \ldots, \mathsf d ( x_m, y_m) \} < \infty \,.
\]
Moreover, if $\mathsf c_{\mathsf d} ( \pi (x_i) ) \ge  \mathsf d (x_i , y_i)$ for every $i \in [1,m]$, then $\mathsf c_{\mathsf d} (H) = \max \{ \mathsf d (x_1 , y_1), \ldots, \mathsf d ( x_m, y_m) \}$.
\end{proposition}

\begin{proof}
It is sufficient to prove the asserted equation for   $ \mathsf c_{\mathsf d} (H)$. Then the upper bound for $\sup \Delta (H)$ follows from \eqref{basic-inequality}.

We choose an element $a \in H$ and two factorizations  $z, z' \in \mathsf Z (H)$  with $\pi (z) = \pi (z') = a$. Since $R = \{(x_1, y_1), \ldots, (x_m,y_m)\}$ generates the congruence relation defining $H$, there are $z=z_0, \ldots, z_n=z' \in \pi^{-1} (a)$ where $z_j$ arises from $z_{j-1}$ by replacing some $x_i$ by some $y_i$ (or conversely) for some $i \in [1,m]$ and all $j \in [1,n]$. Thus $\mathsf d (z_{j-1}, z_j) \le \max \{ \mathsf d (x_1, y_1), \ldots, \mathsf d (x_m,y_m) \}$ and  $\mathsf c_{\mathsf d} (a) \le \max \{ \mathsf d (x_1, y_1), \ldots, \mathsf d (x_m,y_m) \}$ which implies the asserted upper bound for $\mathsf c_{\mathsf d} (H)$..

To verify the moreover statement, let $\pi \colon \mathsf Z (H) \to  H$ denote the factorization homomorphism, and for every $i \in [1,m]$ we set $a_i = \pi (x_i)$. Then $x_i, y_i$ are factorizations of $a_i$,  and the assumption  implies that
\[
\mathsf d (x_i, y_i) \le \mathsf c_{\mathsf d} (a_i) \le \max \{ \mathsf d (x_1 , y_1), \ldots, \mathsf d ( x_m, y_m) \} \,,
\]
and hence we obtain that $\mathsf c_{\mathsf d} (H) = \max \{ \mathsf d (x_1 , y_1), \ldots, \mathsf d ( x_m, y_m) \}$.
\end{proof}

\medskip
\begin{proposition} \label{3.5}
Let $H$ be a reduced atomic unit-cancellative monoid with $H \ne \{1\}$ and suppose that any of the following conditions hold{\rm \,:}
\begin{enumerate}
\item[(a)] There is an $L \in \mathcal L (H)$ such that $\rho (L)=\rho (H) < \infty$.

\smallskip
\item[(b)]   $H$ is cancellative and  its monoid of relations is finitely generated.

\smallskip
\item[(c)] $H$ is commutative and finitely generated.
\end{enumerate}
Then  there  is a bound $M \in \N$ such that $\rho_{k}(H) - \rho_{k-1}(H) \le M$ for all $k \in \N_{\ge 2}$.
\end{proposition}

\noindent
{\it Remark.} Since $\rho_1 (H)=1$, the existence of a bound $M$ as above implies in particular that $\rho_k (H) < \infty$ for all $k \in \N_{\ge 2}$.

\begin{proof}
(a) Since $H$ is a unit-cancellative monoid, the system of sets of lengths $\mathcal L (H)$ is a directed family. Thus if the elasticity is accepted, then the assertion on the growth behavior on the $\rho_k (H)$ follows from an associated statement in the setting of directed families (see \cite[Proposition 2.8]{F-G-K-T17}).

\smallskip
(b) It is sufficient to verify that Condition (a) holds.  By definition, we have
\[
\begin{aligned}
\rho (H) & = \{ \rho (L) \mid L \in \mathcal L (H) \}
  = \{ \sup \mathsf L (a) / \min \mathsf L (a) \mid a \in H \} \\
 & = \{ |x|/|y| \mid (x,y) \in \ \sim_H \} \,.
\end{aligned}
\]
Suppose that $A \subset \ \sim_H$ is a finite minimal generating set (note that $A$ is symmetric, whence $(x,y) \in A$ implies that $(y,x) \in A$). We assert that
\begin{equation} \label{rho-eq}
\rho (H) = \max \{ |x'|/|y'| \mid (x',y') \in A \} \,.
\end{equation}
We show that $(x,y) \le \max \{ |x'|/|y'| \mid (x',y') \in A \}$ for all $(x,y) \in \ \sim_H$ and proceed by induction on $|x|+|y|$. If $(x,y) \in A$, then the assertion holds. Suppose that $(x,y) \notin A$. Then $(x,y)$ is a product of two elements from $\sim_H$ where both are distinct from the identity element, say $(x,y) = (x_1x_2, y_1y_2)$ where $(x_1, y_1), (x_2, y_2) \in \ \sim_H$ and $|x_{\nu}|+|y_{\nu}| < |x|+|y|$ for $\nu \in [1,2]$. Then the induction hypothesis implies that
\[
\frac{|x|}{|y|} = \frac{|x_1|+|x_2|}{|y_1|+|y_2|} \le \max \big\{ \frac{|x_1|}{|y_1|}, \frac{|x_2|}{|y_2|} \big\} \le \max \{ |x'|/|y'| \mid (x',y') \in A \} \,.
\]
Thus \eqref{rho-eq} holds and hence there is an $L \in \mathcal L (H)$ such that $\rho (L)=\rho (H)$.

(c) This follows from \cite[Proposition 3.4]{F-G-K-T17}.
\end{proof}

\medskip
\begin{theorem} \label{3.6}
Let $H$  be a reduced atomic unit-cancellative monoid with $H \ne \{1\}$. If any of the following conditions hold,  then there  is a bound $M \in \N$ such that $\rho_{k}(H) - \rho_{k-1}(H) \le M$ for all $k \in \N_{\ge 2}$ and $\mathcal L (H)$ satisfies the Structure Theorem for Unions.
\begin{enumerate}
\item[(a)] $H$ is finitely presented and there is an $L \in \mathcal L (H)$ such that $\rho (L)=\rho (H) < \infty$.

\smallskip
\item[(b)]   $H$ is cancellative and  its monoid of relations is finitely generated.

\smallskip
\item[(c)] $H$ is commutative and finitely generated.
\end{enumerate}
\end{theorem}

\begin{proof}
If Condition (b) holds, then $H$ is finitely presented as observed above and if Condition (c) holds, then $H$ is finitely presented by Redei's Theorem (\cite[Chapter VI.1]{Gr01}). Thus in all three cases the set of distances $\Delta (H)$ is finite by
Proposition \ref{3.4}.  Thus  Condition (b) in Lemma \ref{3.1} holds by Proposition \ref{3.5} and hence $\mathcal L (H)$ satisfies the Structure Theorem for Unions.
\end{proof}

\smallskip
An overview of  commutative monoids and rings satisfying the Structure Theorem for Unions can be found in \cite[Section 3]{F-G-K-T17}. Based on these results it follows that the Structure Theorem for Unions holds for (not necessarily commutative) transfer Krull monoids of finite type \cite[Theorem 13]{Ge16c} (the systems of sets of lengths of such monoids coincide with the corresponding systems of commutative Krull monoids).  Crucial examples of transfer Krull monoids and domains are due to N. Baeth and D. Smertnig \cite{Sm13a, Ba-Ba-Go14, Ba-Sm15, Ba-Je16a, Sm17a}, and for an overview we refer to \cite[Section 4]{Ge16c}.

\smallskip
\begin{remark} \label{3.7}~

\smallskip
1. There is a commutative  Krull monoid $H$ with finite set of distances and finite $k$th elasticities $\rho_k (H)$ for all $k \in \N$ whose system $\mathcal L (H)$ does not satisfy the Structure Theorem for Unions. In this case there is no $M \in \N$ such that $\rho_{k}(H) - \rho_{k-1}(H) \le M$ for all $k \in \N_{\ge 2}$ (\cite[Theorem 4.2]{F-G-K-T17}).

\smallskip
2. A monoid $H$ is called {\it almost commutative} (see \cite[Section 6]{Ba-Sm15}) if being associated  is a congruence relation on $H$ and the associated reduced monoid is commutative.  Thus, if $H$ is almost commutative and $H_{\red}$ is finitely generated, then $\mathcal L (H) = \mathcal L (H_{\red})$ satisfies the Structure Theorem for Unions by Condition (c) of Theorem \ref{3.6}.

A monoid $H$ is said to be {\it normalizing} if $aH = Ha$ for all $a \in H$. Normalizing monoids play an important role in the study of semigroup algebras (see \cite{Je-Ok07a, Ak-Ma16a, Ok16a}) and normalizing Krull monoids are almost commutative (\cite{Ge13a}).

\smallskip
3. Let $H$ be as in Theorem \ref{3.6}, namely a reduced atomic unit-cancellative monoid. We compare Conditions (a), (b), and (c).
As shown in Proposition \ref{3.5}, Condition (b) implies Condition (a), and if $H$ is cancellative, then Condition (c) implies Conditon (b) by Lemma \ref{3.3}.3. However,  none of the reverse implications is true in general. A commutative finitely generated monoid without accepted elasticity can be found in \cite[Remarks 3.11]{F-G-K-T17}.

The monoids $B_n$ discussed in Section \ref{4} satisfy the conclusions of Theorem \ref{3.6} but none of the Conditions (a) - (c). On the other hand, it is easy to see that there are cancellative monoids with a single defining relation, whose sets of distances are finite by Proposition \ref{3.4} but whose $k$th elasticities are infinite for all $k \ge 2$. For example, consider the monoid $H = \langle a, b \mid a^2 = ba^2b \rangle$. Then $H$ is reduced atomic with $\mathcal A (H) = \{a,b\}$, and $H$ is cancellative because it is an Adyan monoid. Clearly, $\rho_2 (H) = \infty$ and hence $\rho_k (H) = \infty$ for all $k \ge 2$.

\smallskip
4. In a forthcoming paper \cite{Ba-Sm18}, Baeth and Smertnig verify the Structure Theorem for Unions for local quaternion orders. Their result and the present Theorem \ref{3.6}  reveal the first non-commutative monoids $H$ for which it could be shown that $\mathcal L (H)$ satisfies the Structure Theorem for Unions  without showing that $\mathcal L (H)=\mathcal L (B)$ for some  commutative monoid $B$ (see also Corollary \ref{4.4}).
\end{remark}

\medskip
\section{Sets of lengths of the monoid $B_{n}= \langle a,b\ |\ ba=b^{n} \rangle$} \label{4}
\medskip

For $n \in \N_0$, consider the monoid
\[
B_{n}=  \langle a,b\ |\ ba=b^{n} \rangle \,.
\]
Then $B_0$ is the bicyclic monoid, and $B_1$ is the submonoid of right units of Warne's 2-dimensional bicyclic monoid.
If $a^{k}b^{m}\in{B_{0}}$, then $a^{k}b^{m}= a^{k}b^{m}ba=a^{k}b^{m}baba=\cdots$, and if $a^{k}b^{m}\in{B_{1}}$ with $m\neq0$, then $a^{k}b^{m}= a^{k}b^{m}a=a^{k}b^{m}aa=\cdots$. Therefore, $B_{0}$ and $B_{1}$ are not BF-monoids.

Suppose that  $n \ge 2$. In this case $B_n$ is a M\"obius monoid (\cite[Secton 2.1]{Schw15a}). Observe that multiplication in $B_{n}$ is given by
$$a^{k}b^{m}\cdot{a^{r}b^{s}}=\left\{\begin{array}{ccl}a^{k+r}b^{s}&\text{if}\quad m=0
 \\a^{k}b^{m+(n-1)r+s}&\text{if}\quad m>0.\end{array}\right.$$
Clearly, the monoid $B_n$ is reduced, unit-cancellative, right cancellative, but not left cancellative, and  atomic with
$\mathcal{A}(B_{n})=\{a,b\}$.

\medskip
\begin{theorem} \label{4.1}
Let $n \in \N_{\ge 2}$.
\begin{enumerate}
\item If $\mathsf d$ denotes the Levenshtein distance, then $\mathsf c_{\mathsf d} (B_n) = n-1$, and $B_n$ is half-factorial if and only if $n=2$.

\smallskip
\item For every $k,m \in \N_0$ we have
\[
\mathsf L (a^kb^m) = k+m-q_{m,n}(n-2) + (n-2) \cdot [0,q_{m,n}] \,,
\]
where
\[
q_{m,n}=
\begin{cases}
\big\lfloor \frac{m}{n-1} \big\rfloor   & \text{if} \quad  (n-1) \nmid m \,, \\
\frac{m}{n-1}-1 & \text{if}\quad (n-1) \mid m \ \text{and}\ m \neq 0 \,, \\
0                & \text{if}\quad m=0 \,.
\end{cases}
\]
\end{enumerate}
\end{theorem}

\begin{proof}
Since $\mathsf d (ab, b^n) = n-1$, Proposition \ref{3.4} implies that $\mathsf c_{\mathsf d} (B_n) = n-1$.

Let $k,m \in \N_0$. If $m=0$, then $\mathsf L (a^k) = \{k\}$  and hence the claim holds. Suppose that $m>0$. To begin with, consider two elements $u, v \in \mathcal F^* (\{a,b\})$.  If $v$ is directly derivable from $u$ (that is $v=w_{1}b^{n}w_{2}$ and $u=w_{1}baw_{2}$ or $v=w_{1}baw_{2}$ and $u=w_{1}b^{n}w_{2}$), then $|v|=|u|\pm(n-2)$. This shows that $\Delta \big( \mathsf L (a^kb^m) \big) \subset \{n-2\}$, and hence $B_n$ is half-factorial if and only if $n=2$.

For integers $m$ and $n$ ($m\ge 0$, $n\ge 2$), we define $q_0, q_1, \ldots$ and $r_0, r_1, \ldots$ as quotients and remainders in the following sequence of divisions (with $q_i$ being the first quotient which vanishes):
\begin{itemize}
\item[ ] $m=nq_{0}+r_{0}$\ \ \ ($r_{0}<n;\ \ q_{0}\neq0$)
\item[ ] $q_{0}+r_{0}=nq_{1}+r_{1}$\ \ \ ($r_{1}<n;\ \ q_{1}\neq0$)
\item[ ] $q_{1}+r_{1}=nq_{2}+r_{2}$\ \ \ ($r_{2}<n;\ \ q_{2}\neq0$)
\item[ ] \ \ \ .......................
\item[ ] $q_{i-1}+r_{i-1}=n\cdot0+r_{i}$\ \ \ ($r_{i}<n;\ \ 0=q_{i}$) \,,
\end{itemize}
and we set $q_{m,n} = q_0+q_1+ \ldots + q_i$. If $m \ge n$, then
\[
\begin{aligned}
a^{k}b^{m} & = a^{k}b^{m-(n-1)}a=a^{k}b^{m-2(n-1)}a^{2}=\cdots=a^{k}b^{m-q_{0}(n-1)}a^{q_{0}}= \ldots \\
 & =a^{k}b^{m-(q_{0}+q_{1})(n-1)}a^{q_{0}+q_{1}}=\cdots=a^{k}b^{m-q_{m,n}(n-1)}a^{q_{m,n}} \,.
\end{aligned}
\]
This shows that $k+m-q_{m,n}(n-2) + (n-2) \cdot [0,q_{m,n}] \subset \mathsf L (a^kb^m)$. Conversely, by construction and the fact that $\Delta \big( \mathsf L (a^kb^m) \big) \subset \{n-2\}$, it follows that any other factorization length of $a^kb^m$ is contained in $k+m-q_{m,n}(n-2) + (n-2) \cdot [0,q_{m,n}]$.

It remains to show that $q_{m,n}$ has the asserted value. By adding terms in the columns in the above sequence of divisions (i.e., in the defining table  of $q_0, q_1, \ldots$ and $r_0, r_1, \ldots$) we see that
\[
q_{m,n} = q_0 + \ldots + q_i = \frac{m-r_i}{n-1} \,.
\]
Note that $m>0$ implies $r_i>0$. If  $(n-1)\mid m$ then $r_{i}=n-1$ whence
$$q_{m,n}=\frac{m}{n-1}-1.$$
If  $(n-1) \nmid m$,  then $0<\frac{r_{i}}{n-1}<1$ and
\[
q_{m,n}=\frac{m}{n-1}-\frac{r_{i}}{n-1}=\Big{\lfloor}\frac{m}{n-1}\Big{\rfloor} \,. \qedhere
\]
\end{proof}

\smallskip
If $n=2$, then $B_n$ is half-factorial whence $\mathcal L (B_n) = \big\{ \{k\} \mid k \in \N_0 \big\}$. Thus in our further study of $\mathcal L (B_n)$ we always suppose that $n \ge 3$.

\medskip
\begin{theorem} \label{4.2}
Let $n \in \N_{\ge 3}$. Then
\[
\mathcal L (B_n) = \{ x + (n-2) \cdot [0,q] \mid x,q \in \N_0 \ \text{with} \ x=q=0 \ \text{or} \ x>q \} \,.
\]
\end{theorem}

\begin{proof}
By definition, $\mathsf L (1) = \{0\}$ and this is the only $L \in \mathcal L (B_n) $ with $0 \in L$. First we show that $\mathcal L (B_n)$ is contained in the set on the right hand side. Let $k,m \in \N_0$ with $k+m>0$. By Theorem \ref{4.1}, we have
\[
\mathsf L (a^kb^m) = k+m-q_{m,n}(n-2) + (n-2) \cdot [0, q_{m,n}] \,,
\]
and so we need to  verify that
\[
k+m-q_{m,n}(n-2) > q_{m,n} \,. \tag{$*$}
\]
If $q_{m,n}=0$, then $(*)$ holds. Otherwise, we  have
\[
\frac{m-n}{n-1}<q_{m,n}<\frac{m}{n-1}
\]
which implies that $0 < m - q_{m,n}(n-1) < n$ and thus $(*)$ holds.

Conversely, let $x,q \in \N_0$ with $x>q$. We choose $k \in \N_0$ such that $0 < x-q-k < n$ and set $m = q(n-1)+x-q-k$. If $n-1 \mid m$, then $x-q-k=n-1$ and $q_{m,n}=q$ whence
\[
\begin{aligned}
\mathsf L (a^kb^m) & = k+m-q_{m,n}(n-2)+(n-2) \cdot [0, q_{m,n}] \\
 & = x + (n-2) \cdot [0,q] \,.
\end{aligned}
\]
If $n-1 \nmid m$, then again we obtain that $q_{m,n} = \lfloor \frac{m}{n-1} \rfloor = q$ and hence
\[
\mathsf L (a^kb^m)  = k+m-q_{m,n}(n-2)+(n-2) \cdot [0, q_{m,n}]  = x + (n-2) \cdot [0,q] \,. \qedhere
\]
\end{proof}

\medskip
\begin{corollary} \label{4.3}
Let $n \in \N_{\ge 3}$.
\begin{enumerate}
\item $\rho (B_n) = n-1 > \rho (L) $ for every $L \in \mathcal L (B_n)$.

\smallskip
\item For every $\ell \ge 2$ we have
      \[
      \mathcal U_{\ell} (B_n) = \ell - q_{\ell,n} (n-2) + (n-2)\cdot [0, q_{\ell,n}+\ell-1] \,.
      \]
\end{enumerate}
\end{corollary}

\begin{proof}
1. Using the explicit description of $\mathcal L (B_n)$ given in Theorem \ref{4.2}, we infer that
\[
\begin{aligned}
\rho (B_n) & = \sup \{ \max L / \min L \mid \{0\} \ne L \in \mathcal L (B_n) \} \\
 & = \sup \Big\{ \frac{x+(n-2)q}{x} \mid x,q \in \N_0, x>q \Big\} = n-1 > \frac{\max L}{ \min L}
\end{aligned}
\]
for every $L \in \mathcal L (B_n)$.

\smallskip
2. Let $\ell \ge 2$. By Theorem \ref{4.2}, $\mathcal U_{\ell} (B_n)$ is a union of arithmetical progressions having  difference $n-2$ and containing the element $\ell$. Thus $\mathcal U_{\ell} (B_n)$ is an arithmetical progression with difference $n-2$ which contains the element $\ell$. The maximum of $\mathcal U_{\ell} (B_n)$ stems from a set of the form $x + (n-2) \cdot [0,q]$ with $x>q$ and $\ell \in x + (n-2) \cdot [0,q]$ whence $\max \mathcal U_{\ell} (B_n) = \ell + (n-2) ( \ell-1)$. In order to determine $\min \mathcal U_{\ell} (B_n)$ we have to find the maximal $y \in \N$ such that
\[
\ell - y(n-2) +(n-2)\cdot [0,y] = \ell \quad \text{and} \quad y < \ell - y (n-2)
\]
which implies $y < \ell/(n-1)$ and thus $y = q_{\ell, n}$.
\end{proof}

\bigskip
Let $G$ be an additive abelian group and $G_0 \subset G$ a subset. We introduce a commutative Krull monoid having a combinatorial flavor. Because of connections to Additive Combinatorics, the elements $S \in \mathcal F (G_0)$ are called sequences over $G_0$. Let $\sigma \colon \mathcal F (G_0) \to G$ be the homomorphism defined by $\sigma (g) = g$ for every $g \in G_0$. Thus, if $S = g_1 \cdot \ldots \cdot g_{\ell} \in \mathcal F (G_0)$, then $\sigma (S) = g_1 + \ldots + g_{\ell} \in G$ denotes its sum. Then
\[
\mathcal B (G_0) = \{ S \in \mathcal F (G_0) \mid \sigma (S)=0 \} \subset \mathcal F (G_0)
\]
is a submonoid, called the {\it monoid of zero-sum sequences} over $G_0$. Clearly, $\mathcal B (G_0) \subset \mathcal F (G_0)$ is a saturated submonoid whence $\mathcal B (G_0)$ is a commutative Krull monoid (see the definition of Krull monoids given in Section \ref{2}).

We recall the concept of weak transfer homomorphisms (cf. \cite[Definition 3.1]{Ba-Ba-Go14} and \cite[Definition 2.1]{Ba-Sm15}). Let $H$ and $B$ be atomic unit-cancellative monoids. A monoid homomorphism $\theta \colon H \to B$ is called a {\it weak transfer homomorphism} if the following two properties are satisfied:
\begin{itemize}
\item[{\bf (T1)}] $B = B^{\times} \theta (H) B^{\times}$ and $\theta^{-1} (B^{\times})=H^{\times}$.
\item[{\bf (WT2)}] If $a \in H$, $n \in \N$, $v_1, \ldots, v_n \in \mathcal A (B)$ and $\theta (a) = v_1 \cdot \ldots \cdot v_n$, then there exist $u_1, \ldots, u_n \in \mathcal A (H)$ and a permutation $\tau \in \mathfrak S_n$ such that $a = u_1 \cdot \ldots \cdot u_n$ and $\theta (u_i) \in B^{\times} v_{\tau (i)} B^{\times}$ for each $i \in [1,n]$.
\end{itemize}
If $\theta \colon H \to B$ is a weak transfer homomorphism, then it is easy to check that $\theta (\mathcal A (H))=\mathcal A (B)$ and  $\mathcal L (H)=\mathcal L (B)$. An atomic unit-cancellative monoid $H$ is said to be a {\it transfer Krull monoid} if one of the following two equivalent properties is satisfied:
\begin{itemize}
\item[(a)] There is a commutative Krull monoid $B$ and a weak transfer homomorphism $\theta \colon H \to B$.

\item[(b)] There is an abelian group $G$, a subset $G_0 \subset G$, and a weak transfer homomorphism $\theta \colon H \to \mathcal B (G_0)$.
\end{itemize}
In case (b) we say that $H$ is a transfer Krull monoid over $G_0$. Thus by definition, every commutative Krull monoid is transfer Krull but there is an impressive list of transfer Krull monoids which are not commutative Krull (for a survey see  \cite[Section 4]{Ge16c}). Since the class of commutative Krull monoids is huge and since for most classes of rings and monoids only qualitative finiteness or infiniteness results for arithmetical invariants are known but no precise values or explicit descriptions (such as the one given in Theorem \ref{4.2}), we know only for small classes of monoids that they are not transfer Krull and all of them are  commutative (see \cite{Ge-Sc-Zh17b}).

\medskip
\begin{corollary} \label{4.4}
Let $n \in \N_{\ge 3}$.
\begin{enumerate}
\item For every reduced atomic cancellative monoid $H$, whose monoid of relations is finitely generated, we have $\mathcal L (H) \ne \mathcal L (B_n)$.

\smallskip
\item The monoid $B_n$ is not a transfer Krull monoid.

\end{enumerate}
\end{corollary}

\begin{proof}
1. Let $H$ be a reduced atomic cancellative monoid, whose monoid of relations is finitely generated. Then $H$ has accepted elasticity by Proposition \ref{3.5} (indeed, in its proof we showed that Condition (b) implies Condition (a)$\,$). Since $B_n$ does not have accepted elasticity by Corollary \ref{4.3}.1, the claim follows.

\smallskip
2. Assume to the contrary that there is an abelian group $G$, a subset $G_0 \subset G$, and a weak transfer homomorphism $\theta \colon B_n \to \mathcal B (G_0)$. Since the set of atoms $\mathcal A (B_n)$ is finite, the set of atoms
\[
\mathcal A \big( \mathcal B (G_0) \big) = \theta \big( \mathcal A (B_n) \big)
\]
is finite. Thus the set
\[
G_1 := \bigcup_{U \in \mathcal A ( \mathcal B (G_0))} \{ g \in G_0 \mid \mathsf v_g (U)>0 \} \ \subset G_0
\]
is finite and $\theta (B_n) \subset \mathcal B (G_1)$. Thus we have a weak transfer homomorphism $\theta \colon B_n \to \mathcal B (G_1)$ and hence $\mathcal L (B_n) = \mathcal L \big( \mathcal B (G_1) \big)$. Since $G_1$ is finite, $\mathcal B (G_1)$ is finitely generated by \cite[Theorem 3.4.2]{Ge-HK06a}. Clearly, $\mathcal B (G_1)$ is reduced atomic cancellative, and since it is finitely generated, its monoid of relations is finitely generated by Lemma \ref{3.3}.3, a contradiction to  1.
\end{proof}

\medskip
Let $H$ be an atomic unit-cancellative monoid with $H \ne H^{\times}$. Then, as already mentioned,  for every $L_1, L_2 \in \mathcal L (H)$ there is an $L \in \mathcal L (H)$ such that $L_1+L_2 \subset L$ (clearly, if $L_i = \mathsf L (a_i)$ for $i \in [1,2]$, then $\mathsf L (a_1)+\mathsf L (a_2) \subset \mathsf L (a_1a_2)$). We say that $\mathcal L (H)$ is {\it closed under set addition} if for every $L_1, L_2 \in \mathcal L (H)$ we have $L_1+L_2 \in \mathcal L (H)$.  Whereas the first property holds for all atomic unit-cancellative monoids (indeed, all $\mathcal L (H)$ are directed families), the property of being closed under set addition is extremely restrictive. Clearly, if $H$ is a BF-monoid and $\mathcal L (H)$ is closed under set addition, then $\mathcal L (H)$ itself is a reduced atomic unit-cancellative monoid with set addition as operation and with zero  element $\mathsf L (1) = \{0\}$.

The system $\mathcal L (H)$ is closed under set addition in both extremal cases, namely that either
$H$ is half-factorial (in this case we have $\mathcal L (H) = \big\{ \{k\} \mid k \in \N_0 \big\}$) or that every finite subset $L \subset \N_{\ge 2}$ is a set of lengths (this holds true, among others, for transfer Krull monoids   over infinite abelian groups).
If $H$ is a transfer Krull monoid over a finite abelian group $G$, then $\mathcal L (H)$ is closed under set addition if and only if $\exp (G) + \mathsf r (G) \le 5$, where $\exp (G)$ denotes the exponent  and $\mathsf r (G)$ the rank of $G$ (\cite[Theorem 1.1]{Ge-Sc16b}).

In our final result we reveal that   $\mathcal L (B_{n})$ is closed under set addition, and we determine its monoid theoretical structure.

\medskip
\begin{theorem} \label{4.5}
Let $n \in \N_{\ge 3}$.
\begin{enumerate}
\item $\mathcal L (B_n)$  is closed under set addition.

\smallskip
\item There is a monoid isomorphism
      \[
      \begin{aligned}
      \Phi \colon \mathcal L (B_n)  & \to H \\
                  x+(n-2)\cdot [0,q] & \mapsto (x,q) \,,
      \end{aligned}
      \]
      where $H = \{(k,i) \mid k,i \in \N_0, k=i=0 \ \text{or} \ k>i\} \subset (\N_0^2, +)$. Clearly, $H$  is a reduced atomic commutative cancellative half-factorial monoid, and we have $\mathcal A (H) = \{ (k,k-1) \mid k \in \N \}$.
\end{enumerate}
\end{theorem}

\begin{proof}
By Theorem \ref{4.2}, we have $\mathcal L (B_n) = \{ x+(n-2)\cdot [0,q] \mid x, q \in \N_0, \ x=q=0 \ \text{or} \ x >q \}$.

1. Clearly, $\{0\}$ is the zero element of $\mathcal L (B_n)$ and if $x_1, x_2, q_1,q_2 \in \N_0$ with $x_i > q_i$ for $i \in [1,2]$, then
\[
\big( x_1+(n-2)\cdot [0,q_1] \big) + \big( x_2+(n-2)\cdot [0,q_2] \big) = (x_1+x_2) + (n-2) \cdot [0, q_1+q_2] \ \in \mathcal L (B_n) \,.
\]
Thus $\mathcal L (B_n)$ is closed under set addition and hence a monoid with set addition as operation.

2. Obviously, $\Phi$ is an isomorphism. As a submonoid of the free abelian monoid $(\N_0^2, +)$, $H$ is reduced atomic commutative  cancellative, and clearly we have $\mathcal A (H) = \{ (k,k-1) \mid k \in \N \}$. To show that $H$ is half-factorial, consider an equation of the form
\[
(k,i)=(k_{1},k_{1}-1)+\cdots+(k_{r},k_{r}-1)=(k'_{1},k'_{1}-1)+\cdots+(k'_{s},k'_{s}-1) \,,
\]
which implies that $r=k-i=s$.
\end{proof}

\bigskip
\noindent
{\bf Acknowledgement.} We thank the referees for their careful reading.

\providecommand{\bysame}{\leavevmode\hbox to3em{\hrulefill}\thinspace}
\providecommand{\MR}{\relax\ifhmode\unskip\space\fi MR }
\providecommand{\MRhref}[2]{%
  \href{http://www.ams.org/mathscinet-getitem?mr=#1}{#2}
}
\providecommand{\href}[2]{#2}

\end{document}